\RequirePackage[l2tabu, orthodox]{nag}
\RequirePackage{snapshot}

\documentclass[reqno,final]{amsart}
\usepackage[utf8]{inputenc}

\usepackage{natbib}  %Nature-like bibliography
\usepackage{fancyhdr} %Headers
\usepackage{color} %Color definition
\usepackage{hyperref} %Internal and external links
\usepackage{graphicx} %Graphics inclusion
\usepackage{caption}
\usepackage{subcaption}
\usepackage{url}
\usepackage{mathtools}
\usepackage{arydshln}
\usepackage{pstricks}
\usepackage{amssymb}
\definecolor{aleacolor}{rgb}{0.16,0.59,0.78}
\hypersetup{
breaklinks,
colorlinks=true,
linkcolor=aleacolor,
urlcolor=aleacolor,
citecolor=aleacolor}

\renewcommand{\cite}{\citet}

\theoremstyle{plain}
\newtheorem{theorem}{Theorem}[section]                                          
\newtheorem{proposition}[theorem]{Proposition}                          
\newtheorem{lemma}[theorem]{Lemma}
\newtheorem{corollary}[theorem]{Corollary}

\theoremstyle{definition}
\newtheorem{definition}[theorem]{Definition}
\theoremstyle{remark}

\makeatletter \@addtoreset{equation}{section} \makeatother
\renewcommand\theequation{\thesection.\arabic{equation}}
\renewcommand\thefigure{\thesection.\arabic{figure}}
\renewcommand\thetable{\thesection.\arabic{table}}
\newcommand{\defeq}{\vcentcolon=}
\newcommand*{\doi}[1]{\href{http://dx.doi.org/#1}{doi: #1}}

\def\QED{\mbox{$\square$}}
\def\proof{\noindent{\it Proof:~}}
\def\endproof{\hspace*{\fill}~\QED\par\endtrivlist\unskip}

\begin{document}
\renewcommand*\abstractname{Abstract}
\renewcommand{\figurename}{Figure}
\renewcommand{\tablename}{Table}
\renewcommand{\refname}{References}

\title{A New Information Theoretical Concept: \newline
Information-Weighted Heavy-tailed Distributions}

\author{H.M. de Oliveira and R.J.S. Cintra}

\address{Federal University of Pernambuco, UFPE,\newline
Statistics Department,\newline
CCEN-UFPE, Recife, Brazil.}

\email{\{hmo,cintra\}@de.ufpe.br}
\urladdr{\url{http://arxiv.org/a/deoliveira_h_1.html}}
\urladdr{\url{http://arxiv.org/a/cintra_r_1.html}}

\thanks{This paper is dedicated to Professor Fernando Menezes Campello de Souza (PhD, Cornell), in his forthcoming 70th birthday, who has been steadfast at creating an exceptional atmosphere of interest in Statistics, and whose philosophy had a decisive influence on authors' way of looking the world.}

%\date{2015}

\begin{abstract}

Given an arbitrary continuous probability density function, it is introduced a conjugated probability density, which is defined through the Shannon information associated with its cumulative distribution function. These new densities are computed from a number of standard distributions, including uniform, normal, exponential, Pareto, logistic, Kumaraswamy, Rayleigh, Cauchy, Weibull, and Maxwell-Boltzmann. The case of joint information-weighted probability distribution is assessed. An additive property is derived in the case of independent variables. One-sided and two-sided information-weighting are considered. The asymptotic behavior of the tail of the new distributions is examined. It is proved that all probability densities proposed here define heavy-tailed distributions. It is shown that the weighting of distributions regularly varying with extreme-value index $\alpha>0$ still results in a regular variation distribution with the same index. This approach can be particularly valuable in applications where the tails of the distribution play a major role. 
\end{abstract}

\maketitle

\section{Preliminaries}

Information theory is a subject of relevance in many areas, particularly on Statistic~\cite{MacKay},~\cite{Cover-Thomas}. Given an arbitrary random variable with a continuous probability density function (pdf), $f_X(x)$, we can compute the (Shannon) information amount associated with the event $(X \leq x)$, that is, $F_X(x) \defeq P(X \leq x)$ for each $x \in \mathbb{R}$. This is given by $-\log F_X(x)$.
\begin{definition}\label{def:left}
\textit{(cumulative information pdf) The information-weighted density $f_{IX}(x)$ is defined by}:
\begin{equation} \label{eq:lefttail}
f_{IX}(x) \defeq -f_X(x)\cdot \log F_X(x).
\end{equation}
\endproof
\end{definition}
Let us define an operator $\mathcal{I}\{\cdot\}$, which maps the a probability density $f_X(x)$ into another function ${f_{IX}(x)}=\mathcal{I}\{ f_X(x) \}$ according to Def. \ref{def:left}. This can be interpreted as a probability density pair $f_X(x) \leftrightarrow f_{IX}(x)$ and the new density is the former density, but \textit{weighted} by the information provided by its cumulative distribution. 
In the framework of distribution generalization theory, a mapping that takes a distribution in another 
allows the construction of several new distributions (e.g. \cite{leao}), which is particularly attractive due to the fact that the shape of the new distribution is quite flexible. 
For instance, the beta generalized normal distribution (\cite{cintra2014}) encompasses the beta normal, 
beta Laplace, normal, and Laplace distributions as sub-models. This article is in a scope somewhat similar, providing the generation of new probability distributions. However, noteworthy here is the construction of heavy-tailed distributions, even from distributions that do not hold this attribute.\\
The information-conjugated distribution is denoted by inserting an $I$ before the standard distribution, e.g. for a normal distribution, $X \sim N(0,1) \leftrightarrow IX ~\sim IN(0,1)$. 
(remark: the terms information-conjugated and information-weighted are used interchangeably throughout the paper.) This first property of a conjugated pdf is concerning its support:
\begin{corollary}
\textit{The support of $f_{IX}(\cdot)$ is contained in the support of $f_X(\cdot)$}, i.e. $Supp f_{IX} \subseteq Supp f_X ~~~\Box$.
\end{corollary}
Indeed
\begin{equation} \label{eq:normalizacao}
E (- \log F_X(X) )=-\int_{-\infty }^{\infty } {f_X(x)\cdot \log F_X(x)\operatorname{d}\!x}.
\end{equation}
This expression recalls the original definition of Shannon for the differential entropy of a continuous distribution (see~\cite{Michalowicz}), which is defined by
\begin{equation}
H(X) \defeq -\int_{-\infty }^{\infty } {f_X(x)\cdot \log f_X(x)\operatorname{d}\!x}.
\end{equation}
One of the troubling questions of this setting is the possibility of negative values for $H(X)$. 
This is due to the fact that $f_X(x)$ is not upper bounded by the unit. Replacing now $f_X(\cdot)$ by $F_X(\cdot)$ 
in the argument of the logarithm was our initial motivation as an attempt to address this issue, bearing in mind that  $F_X(x) \leq 1 ~(\forall x)$. 
However, rather to redefine entropy, this always resulted in unitary integral, leading the proposal laid down in this paper. The differential Entropy also has an interesting link with the wavelet analysis~\citep{deO}. We show in the sequel that the integral Eqn~\ref{eq:normalizacao} is always the unity, whatever the original probability density. Thus, the operator $\mathcal{I}\{\cdot\}$ preserves probability densities and the calculation of the area under the curve is an isometry.\\
\begin{proposition}
\textit{$f_{IX}(x)$ is a valid probability density}.
\end{proposition}
\begin{proof}
In order to proof that this is a normalized nonnegative function, we shall prove that:\\
\begin{equation} \nonumber
\begin{split}
& i) ~\forall (x) ~f_{IX}(x) \geq 0.\\
& ii) \int_{-\infty }^{\infty }f_{IX}(x)\operatorname{d}\!x=1.
\end{split}
\end{equation}
\\
We remark first that $-\log F_X(x) \geq 0$, so (i) follows. Then we take
\begin{equation}
I \defeq \int_{-\infty }^{\infty }f_{IX}(x)\operatorname{d}\!x=-\int_{-\infty }^{\infty }  f_X(x)\cdot \log F_X(x)\operatorname{d}\!x,
\end{equation}
which can be rewritten in terms of a Stieltjes integral~\cite{Protter}:
\begin{equation}
I=-\int_{-\infty }^{\infty } \log  F_X(x)dF_X(x).
\end{equation}
Note that $F_X(x)$ is the cumulative probability distribution (CDF) of $X$. By the property of pars integration, we derive:
\begin{equation}
\begin{split}
I=-[\log F_X(+\infty)].F_X(+\infty)+ [\log F_X(-\infty)].F_X(-\infty)+\\
\int_{-\infty }^{\infty } F_X(x)d\log F_X(x),
\end{split}
\end{equation}
so that
\begin{equation}
I=\int_{-\infty }^{\infty } F_X(x)d\log F_X(x)=\int_{-\infty }^{\infty } dF_X(x)=1.
\end{equation}
\end{proof}
It is also straightforward to derive (by simple integration) that the CDF associated with pdf of Def. \ref{def:left} is:
\begin{equation} \label{eq:CDF_IX}
F_{IX}(x)=F_{X}(x). \left [ 1-\log ~F_{X}(x) \right ].
\end{equation}
As expected, $F_{IX}(-\infty)=0$ (since $\lim_{y \to 0 }~y\cdot \log y=0$) and $F_{IX}(+\infty)=1$.\\
\textit{How to model probabilistic events described by long-tailed distributions?} There are relatively few distributions used in this setting (e.g. Cauchy, log-normal, Weilbull, Burr...), highlighting the Pareto distribution. A pleasent reading review of different classes of distributions with heavy tails can be found in~\citep{Werner}. We are concerned particularly with two classes:
\begin{itemize}
\item{\text{class D: subexponential distributions,}}
\item{\text{class C: regular variation with tail index $\alpha>0$.}}
\end{itemize}
We show in the sequel that this paper offers a profuson of new options, primarily concerning the class of subexponential distributions~\citep{Goldie}.

\section{Conjugated Information-Weighted Density Associated with Known Distributions}
Now we compute the conjugated information density associated with selected standard distributions $F_X(x)$ selected in Table \ref{table:tab1} (see~\cite{Walpole}). 
\begin{table}[!h]
\centering
\caption{A few standard continuous probability distributions: nomenclature, density and support.}
\label{table:tab1}
\begin{tabular}{ c c c c} 
\hline 
Ddistribution& $f_X(x)$ & $Supp$\\ \hline
$U(0,1)$&1&$0\leq x\leq 1$\\
$N(0,1)$&$\frac{1}{\sqrt{2\pi }}e^{-x^2/2}$& $-\infty\leq x < +\infty$\\
$E(1)$& $e^{-x}u(x)$ & $0 \leq x < +\infty$\\
$logistic(0,1)$& $\frac{e^{-x}}{\left ( 1+e^{-x} \right )^2}$&$-\infty\leq x < +\infty$\\
$Ray(1)$ & $xe^{-x^2/2}$ & $0\leq x < +\infty$\\  
$Par(\alpha,1)$&$\frac{\alpha}{x^{\alpha+1}}$& $x\geq 1~\alpha>0$\\
$Cau(1)$& $\frac{1}{\pi(1+x^2)}$& $-\infty\leq x < +\infty$\\
$Wei(1,k)$& $kx^{k-1}e^{-x^{k}}$ &$0 \leq x < +\infty~k>0$\\
$Maxw(1)$& $\sqrt{\frac{2}{\pi}}x^2e^{-x^{2}/2}$ & $0 \leq x < +\infty$\\
$Kum(a,b)$ & $abx^{a-1}\left ( 1-x^{a} \right )^{b-1}$ & $a \leq x\leq b$ \\
\\ \hline
\end{tabular}
\end{table}

Plots of the probability density of some new probability distributions $f_{IX}$ presented in Table \ref{table:tab2} are shown in Fig. ~\ref{fig:fig1}. 
The mean and variance of the examined distributions were numerically evaluated and the results are shown in Table \ref{table:tab3}. 
As the weight by information leads to a left skew of the distribution, it is expected that the new statistical mean is smaller than that of the original distribution. 
We promptly check that this happens in every case, as expected.\\
There is just one intersection point between the conjugated distribution pair, which is given by 
\begin{equation}
x^* = F_X^{-1} (e^{-1}).
\end{equation}
%plot 
\begin{figure}
       \centering
       \begin{subfigure}{0.20\textwidth}
       \includegraphics[scale=0.3, height=3.5cm]{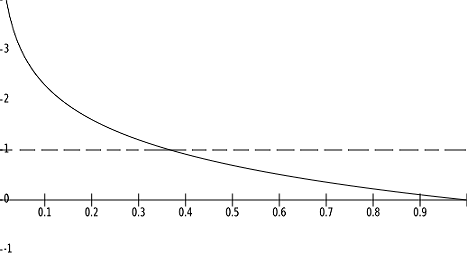}
       \captionof{figure}{IUniform}
       \end{subfigure}%
       \hfill
       \begin{subfigure}{0.2\textwidth}
       \includegraphics[scale=0.5, height=3.3cm]{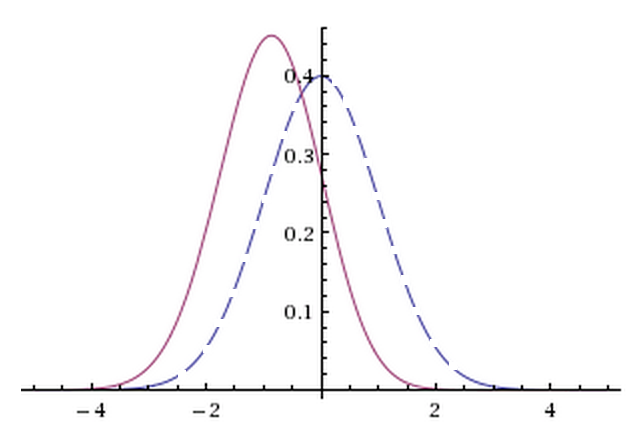}
       \captionof{figure}{Inormal}
       \end{subfigure}

       \begin{subfigure}{0.20\textwidth}
       \includegraphics[height=3.3cm]{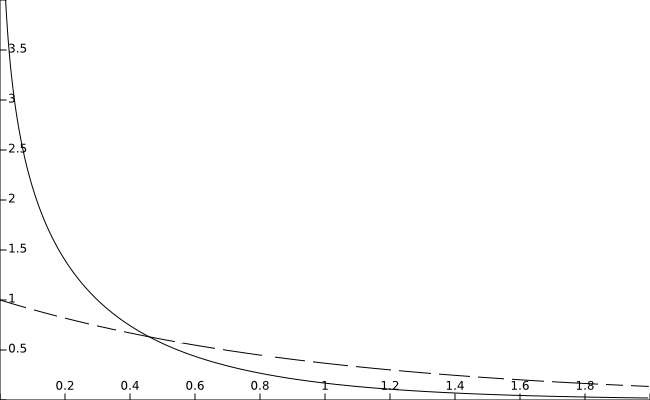}
       \captionof{figure}{IExponential}
       \end{subfigure}%
       \hfill
       \begin{subfigure}{0.2\textwidth}
       \includegraphics[height=3.3cm]{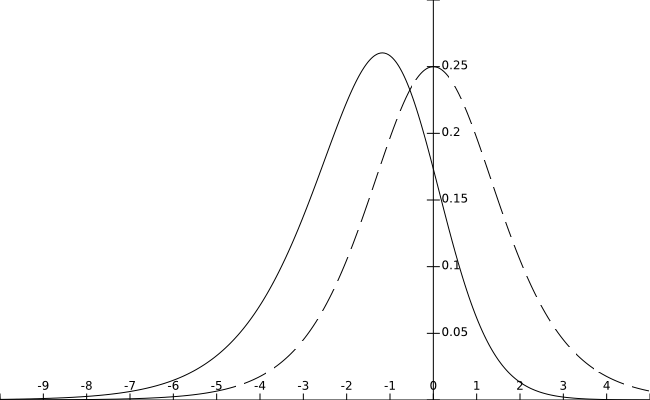}
       \captionof{figure}{Ilogistic}
       \end{subfigure}

       \begin{subfigure}{0.2\textwidth}
       \includegraphics[height=3.3cm]{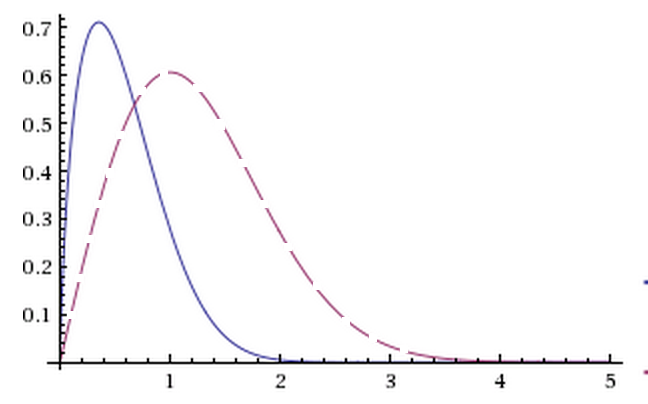}
       \captionof{figure}{IRayleigh}
       \end{subfigure}%
       \hfill
       \begin{subfigure}{0.20\textwidth}
       \includegraphics[height=3.3cm]{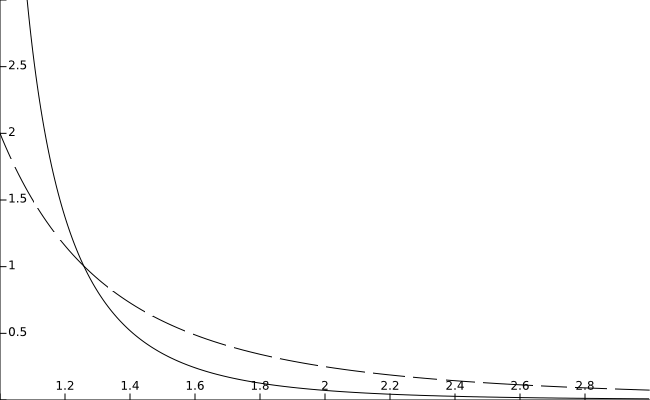}
       \captionof{figure}{IPareto}
       \end{subfigure}

       \begin{subfigure}{0.20\textwidth}
       \includegraphics[height=3.3cm]{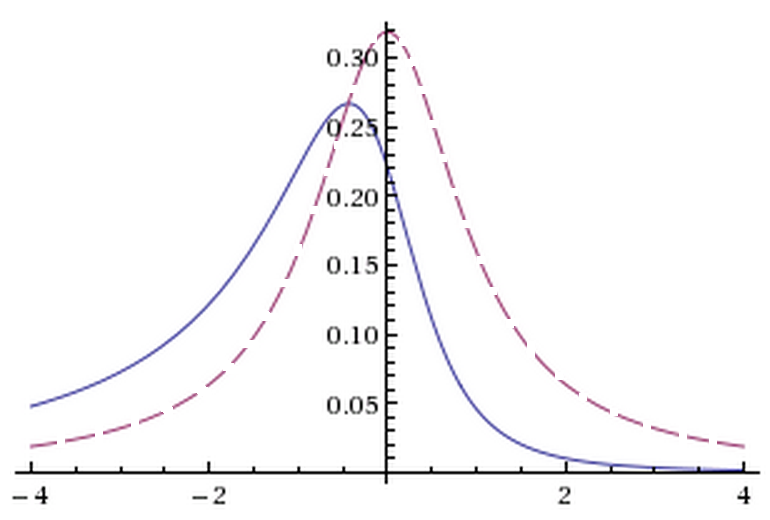}
       \captionof{figure}{ICauchy}
       \end{subfigure}%
        \hfill
       \begin{subfigure}{0.20\textwidth}
       \includegraphics[height=3.3cm]{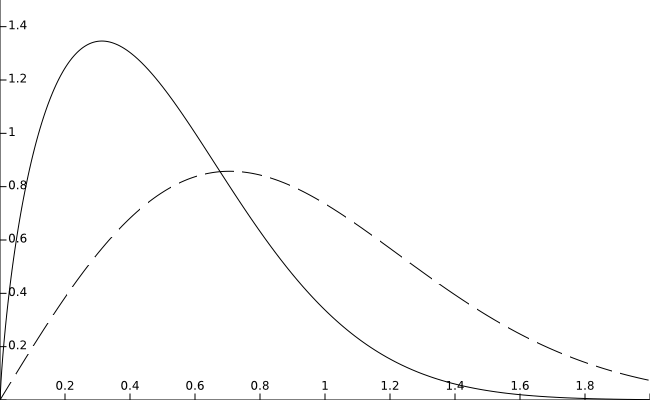}
       \captionof{figure}{IWeibull}
       \end{subfigure}

       \begin{subfigure}{0.2\textwidth}
       \includegraphics[height=3.3cm]{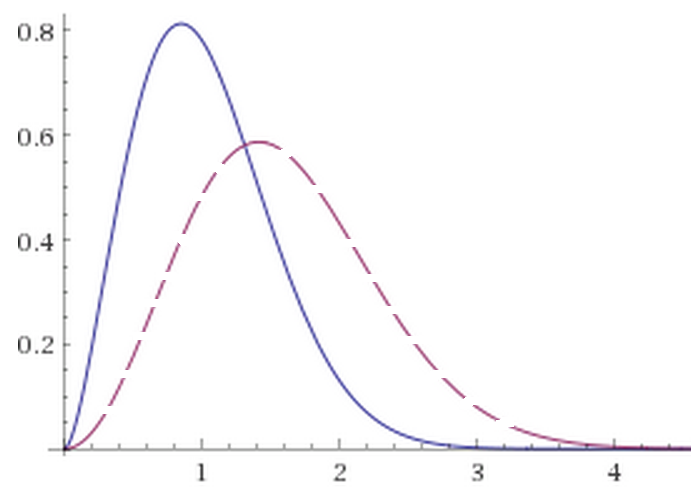}
       \captionof{figure}{IMaxwell}
       \end{subfigure}%
       \hfill
       \begin{subfigure}{0.20\textwidth}
       \includegraphics[height=3.3cm]{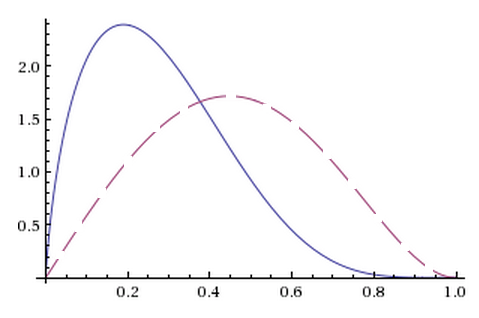}
       \captionof{figure}{IKumaraswamy}
       \end{subfigure}

\caption{Probability density function of information-weighted distribution associated with standard distributions. Plots in solid line: \textit{IUniform$(0,1)$, Inormal$(0,1)$, IExponential$(1)$, Ilogistic$(0,1)$, IRayleigh$(1)$, IPareto $(2,1)$, ICauchy, IWeibull $(1,3)$, IMaxwell-Boltzmann$(1)$ and IKumaraswamy$(2,3)$}. Dashed lines corresponds to original distributions.}
\label{fig:fig1}
\end{figure}
In view of Proposition 1, the improper integrals in Table \ref{table:tab4} are calculated (merely to indorse, all integrals were also computed directly from the Website~\cite{Wolfram}.) 
%Here are some scripts:
\\
%\scriptsize
%\begin{itemize}
%\item{\text{-~integral~from~0~to~+1~ln(x)~dx}}
%\item{\text{-1/sqrt(2 pi)*~integrate~from~-inf~to~inf~{exp(-x$\wedge$2/2)*ln((erf(x/sqrt(2))+1)/2)~}~dx}}
%\item{\text{-~integrate~from~0~to~inf~exp(-x)*ln(1-exp(-x))dx}}
%\item{\text{-~integrate~from~-inf~to~inf~(e$\wedge$(-x)*~ln(e$\wedge$x/(1+e$\wedge$x))/(1+e$\wedge$(-x))$\wedge$2)~dx}}
%\item{\text{-~integrate~from~0~to~inf~x*exp(-x$\wedge$2/2)*ln(1-exp(-x$\wedge$2/2))~dx}}
%\item{\text{-~integrate~from~1~to~inf~{2*ln((1-1/x$\wedge$2))/x$\wedge$3}~dx}}
%\item{\text{-~integrate~from~-inf~to~inf~1/(pi*(1+x$\wedge$2))*ln(tg$\wedge$-1(x)/pi+0.5)~dx}}
%\item{\text{-2*~integrate~from~0~to~inf ~xe$\wedge$(-x$\wedge$2)*ln(1-e$\wedge$(-x$\wedge$2))~dx}}
%\item{\text{-sqrt(2/pi)*~integrate~from~0~to~inf~(x$\wedge$2*exp(-x$\wedge$2/2)*ln(erf(x/sqrt(2))-sqrt(2/pi)*x*exp(-x$\wedge$2/2)))~dx}}
%\item{\text{-~integrate~from~0~to~1~{6x*(1-x$\wedge$2)$\wedge$2*ln(1-(1-x$\wedge$2)$\wedge$3)}~dx}}
%\end{itemize}
%\normalsize
As a measure of central tendency, the mode of the information-weighted distribution can be numerical estimated by determining the roots of the equation:
\begin{equation}
-F_X(x)\cdot \log F_X(x)-\frac{f_X^2(x)}{f'_X(x)}=0.
\end{equation}
Particularly, for the $IN$ distribution, these are roots of the equation:
\begin{equation}\label{eq:roots}
-\frac{e^{-x^2/2}}{\sqrt{2\pi}x}+\left [\frac{erf(\frac{x}{\sqrt{2}})+1}{2} \right ]. \log \left [\frac{erf(\frac{x}{\sqrt{2}})+1}{2} \right ]=0.
\end{equation}
A sketch for the curve defined by Eqn~\ref{eq:roots} is shown in Fig. ~\ref{fig:fig2}. The numerical value was -0.863778 ...
For assessing the asymmetry of the $IN(0,1)$ distribution, their quartiles have been calculated, resulting in $Q_1 \approx-1.4932$, $Q_2\approx-0.8901$, $Q_3\approx-0.2990$. 
Bowley's coefficient of skewness $b_1$ (or Galton's skewness) is defined by (see~\cite{Groeneveld}):
\begin{equation}
b_{1} \defeq \frac{(Q_3-Q_2)-(Q_2-Q_1)}{Q_3-Q_1},
\end{equation}
and thereby $b_1\approx -0.01<0$ (a slight left skew, as expected). Comparable behaviors are observed for other distributions.
In order to have a glimpse about the skewness of the new distribution, a box plot (Fig. ~\ref{fig:boxplot}) have been sketched for the conjugated pair distributed according to  $N(0,1) $ and $IN(0,1) $. Perhaps a beanplot (\cite{kampstra},~\cite{spitzer}) or a plot-balalaika~\cite{shovman} should be convenient when applied to data.
\begin{figure}
       \centering
       \includegraphics[height=7.8cm]{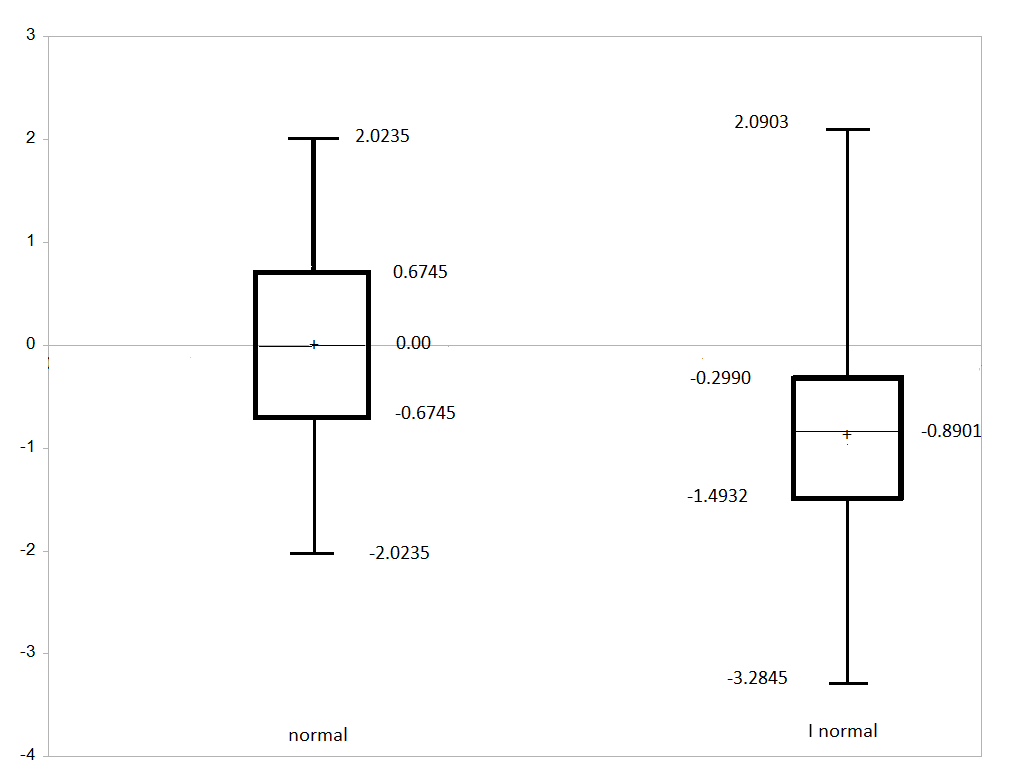}
       \captionof{figure}{Boxplot of data from distributions: normal $N(0,1)$ and Inormal $IN(0,1)$. }.
\label{fig:boxplot}
\end{figure}
\begin{figure}
       \centering
       \includegraphics[height=4.6cm]{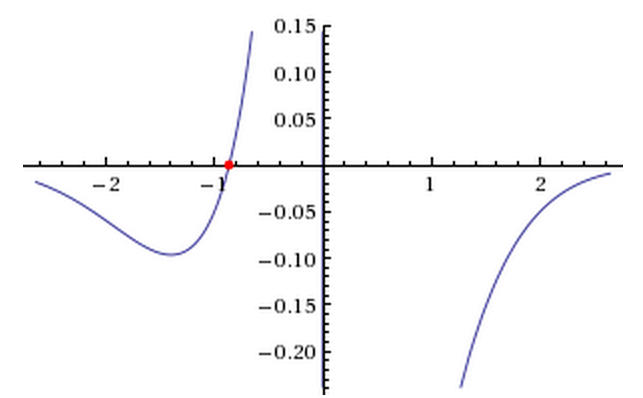}
       \captionof{figure}{Finding the mode of the normal information-weighted distribution (Eqn~\ref{eq:roots}): the root at -0.863778 ... corresponds to the mode of $IX \sim IN(0,1)$.}
\label{fig:fig2}
\end{figure}
\begin{table}[!h]
\centering
\caption{Information-weighted Distributions. Support is the same as in Table \ref{table:tab1}.}
\label{table:tab2}
\begin{tabular}{ c c } 
\hline 
Distribution & $f_{IX}(x)$ \\ \hline
$IU(0,1)$ & $-\log (x)$\\
$IN(0,1)$ & $-\frac{1}{\sqrt{2\pi }}e^{-x^2/2}\log \left ( \frac{erf(\frac{x}{\sqrt{2}})+1}{2} \right )$\\
$IE(1)$ & $-e^{-x} \log \left ( 1-e^{-x} \right )u(x)$ \\
$Ilogistic(0,1)$ & $\frac{e^{-x}\log \left ( \frac{e^{x}}{1+e^{x}} \right )}{\left ( 1+e^{-x} \right )^2}$\\
$IRay(1)$ & $-xe^{-x^2/2}\log \left ( 1-e^{-x^2/2}\right )u(x)$ \\  
$IPar(\alpha,1)$ &$-\frac{\alpha \log \left ( 1-x^{-\alpha} \right )}{x^{\alpha+1}}$\\
$ICau(1)$& $-\frac{1}{\pi(1+x^2)}\log \left ( \frac{tg^{-1}(x)}{\pi} +\frac{1}{2}\right )$ \\
$IWei(1,k)$& $-kx^{k-1}e^{-x^{k}}\log \left ( 1-e^{-x^{k}} \right )u(x)$\\
$IMaxw(1)$& $-\sqrt{\frac{2}{\pi}}x^2e^{-x^{2}/2}\log \left [ erf\left ( \frac{x}{\sqrt{2}}\right )-\sqrt{\frac{2}{\pi}} xe^{-x^{2}/2} \right ]u(x)$ \\
$IKum(a,b)$ & $-abx^{a-1}\left ( 1-x^{a} \right )^{b-1}\log \left [ 1-\left ( 1-x^{a} \right )^{b} \right ]$ \\
\\ \hline
\end{tabular}
\end{table}
The weighted distribution suffers a bias to the left. So we decided to inspect deeper into this issue. Examining the relationship between the statistical mean of the two distributions, it could be demonstrated that:
\begin{proposition}
\textit{Consider a pair $X \leftrightarrow IX$. Their means are related by:}
\begin{equation}\label{eq:mean_shift}
E\left ( IX \right )=E\left ( X \right )+\int_{-\infty }^{+\infty }F_X(x) \log F_X(x) \operatorname{d}\!x.
\end{equation}
The quantity
\begin{equation}
\Delta \mu \defeq -\int_{- \infty }^{+\infty } F_X (x) \log F_X(x) \operatorname{d}\!x
\end{equation} 
is referred to as ``mean shift to the left.''  
\end{proposition}
\begin{proof}
The proof follows from pars integration.\\
\end{proof}
Interestingly, it was discovered subsequently that this definition has a very deep relationship with the \textit{Cumulative Residual Entropy} (Eqn3, \cite{rao2004}). Since the second term of the right-hand side of Eqn~\ref{eq:mean_shift} is negative, then:
\begin{corollary}
\textit{The mean of the information-weighted distribution is shifted to the left with respect to the original mean, i.e.} 
$E\left ( IX \right ) \leq E\left ( X \right ) ~~~\Box$.
\end{corollary}
A novel way to evaluate the mean shift to the left or the cumulative residual entropy \cite{drissi_et_al} can be through a generalization of the approach of information generating function by \cite{Golomb}. Following a rather similar line of reasoning, it can be defined a continuous generator:
\begin{definition}\label{def:generating}
The continuous information generating function, $I(t)$, associated with a CDF $F_X(x)$ is
\begin{equation} \label{eq:generating} 
I(t) \defeq \int_{-\infty }^{+\infty } \{F_X(x)\}^tdx,~~ t \geq 1,
\end{equation}
provided that $\int_{-\infty }^{+\infty } F_X(x)\operatorname{d}\!x < \infty ~~~\Box$. 
\end{definition}
Writing the integrand of Eqn~\ref{eq:generating} in the form of $\{F_X(x)\}^t=exp\left [ t \cdot \log(F_X(x)) \right ]$, it can be derived that
\begin{equation}
-\frac{\partial I(t)}{\partial t} \Bigg|_{t=1} =-\int_{-\infty }^{+\infty}F_X(x) \cdot \log F_X(x)\operatorname{d}\!x.
\end{equation}
Changing the order of the derivative with the improper integral and changing the limit $t \to 1$ with the integral is ensured by applying the Theorems 10.38 and 10.39, pages 282-283 in \cite{Apostol}.\\
Now is quite easy to evaluate the median of the information-weighted distribution:
\begin{equation}
median \approx F_{X}^{-1}(0.186682).
\end{equation}
\centerline{(found with the aid of~\cite{Wolfram}, by typing \textit{roots\{x-x ln x-1/2\}}).}\\
To sum up, given a random variable $X$ with pdf $f_X(x)$ and CDF $F_X(x)$, a new ``conjugated'' ~random variable $IX$ is introduced is this paper, which has pdf $f_{IX}(x)$ (Def. \ref{def:left}) and CDF $F_{IX}(x)$ (Eqn~\ref{eq:CDF_IX}). Featured properties for these new probability distributions are analyzed in Section 5.

\begin{table}[!h]
\centering
\caption{Mean and variance (approximated values) for some information-weighted probability densities. Values for the original distribution also presented as a reference.}
\label{table:tab3}
\begin{tabular}{ c c c } 
\hline 
distribution & mean & variance \\ \hline
$IU(0,1)$ & 0.25 (1/4)& 0.048611 (7/144)\\
$U(0,1)$  & 0.5 (1/2)  & 0.083333 (1/12)\\ 
$IN(0,1)$ & -0.903197 & 0.779875\\
$N(0,1)$ & 0 & 1\\
$IE(1)$ & 0.355066 ($2-\pi^2/6$)& 0.179946\\
$E(1)$ & 1 & 1 \\
$Ilogistic$ & -1.644934 ($-\pi^2/6$)& 2.988185\\
$logistic$ & 0 & 3.289868 ($\pi^2/3$)\\
$IRay(1)$ & 0.716437 & 0.196849\\ 
$Ray(1)$  & 1.253314 ($\sqrt{\pi/2}$) & 0.4292034 ($\frac {4-\pi} {2}$)\\
$IPar(2,1)$ & 1.227411 & 0.138392\\
$Par(2,1)$ & 2 ($\frac {\alpha} {\alpha-1}$) & undefined\\ 
$ICau(1)$ & undefined & undefined\\
$Cau(1)$ & undefined & undefined\\
$IWei(1,2)$ & 0.506598 & 0.098424\\
$Wei(1,2)$  & 0.886227 ($\sqrt{\pi}/2$)& 0.214602\\
$IMaxw(1)$ & 1.02814 & 0.239568\\
$Maxw(1)$ & 1.595769 ($2\sqrt{(2/\pi)}$) & 0.453521 ($\frac {3\pi-8} {\pi}$)\\
$IKum(2,3)$ & 0.278825 & 0.025917\\
$Kum(2,3)$  & 0.457143 (16/35) & 0.25 (1/4)\\
\\ \hline
\end{tabular}
\end{table}

\begin{table}[!h]
\centering
\caption{(Improper) integral calculations derived from information-weighted distributions.} \label{table:tab4}
\begin{tabular}{ c } 
\hline 
$\int_{0 }^{1 }\log (x)\operatorname{d}\!x=-1$\\
$\int_{-\infty }^{+\infty } \frac{1}{\sqrt{2\pi }}e^{-x^2/2}\log \left ( \frac{erf(\frac{x}{\sqrt{2}})+1}{2} \right )\operatorname{d}\!x=-1$ \\
$\int_{0 }^{+\infty }e^{-x} \log \left ( 1-e^{-x} \right )\operatorname{d}\!x=-1$\\
$\int_{-\infty }^{+\infty }\frac{e^{-x}\log \left ( \frac{e^{x}}{1+e^{x}} \right )}{\left ( 1+e^{-x} \right )^2}\operatorname{d}\!x=-1$\\
$\int_{0 }^{+\infty } xe^{-x^2/2}\log \left ( 1-e^{-x^2/2}\right )\operatorname{d}\!x=-1$\\
$\int_{1 }^{+\infty }\frac{\alpha \log \left ( 1-x^{-\alpha} \right )}{x^{\alpha+1}}\operatorname{d}\!x=-1,~\alpha>0$\\
$\int_{-\infty }^{+\infty }\frac{1}{\pi(1+x^2)}\log \left ( \frac{tg^{-1}(x)}{\pi} +\frac{1}{2}\right )\operatorname{d}\!x=-1$\\
$k\int_{0 }^{+\infty }x^{k-1}e^{-x^{k}}\log \left ( 1-e^{-x^{k}} \right )\operatorname{d}\!x=-1,~k>0$\\
$\int_{0 }^{+\infty }\sqrt{\frac{2}{\pi}}x^2e^{-x^{2}/2}\log \left [ erf\left ( \frac{x}{\sqrt{2}}\right )-\sqrt{\frac{2}{\pi}} x^2e^{-x^{2}/2} \right ]\operatorname{d}\!x=-1$\\
$ab\int_{0 }^{1 }x^{a-1}\left ( 1-x^{a} \right )^{b-1}\log \left [ 1-\left ( 1-x^{a} \right )^{b} \right ]\operatorname{d}\!x=-1, ~a,b>0$ \\
\\ \hline
\end{tabular}
\end{table}

\section{Joint Information-Weighted Distribution}
It is straightforward to extend this approach of inducing distributions to the case of joint random variables~\cite{Montgomery}. 
\begin{definition}
\textit{(joint information-weighted pdf) The joint information-weighted density $f_{IX,IY}(x,y)$ of two random variables $X$ and $Y$ is defined by}:
\begin{equation} \label{eq:joint_pdf}
f_{IX,IY}(x,y) \defeq -\frac{1}{2}f_{X,Y}(x,y)\cdot \log F_{X,Y}(x,y). ~~~\Box
\end{equation}
\end{definition}	
The previous equation defines a joint probability density. First, we show:\\
\begin{lemma} \label{marginal}
The marginal probability densities associated with joint density $f_{IX,IY}$ are respectively:
\begin{subequations}
\begin{equation}\label{eq:eqA}
\int_{-\infty }^{\infty }f_{IX,IY}(x,y)dy=\frac{f_{IX}(x)}{2}+\frac{f_{X}(x)}{2},\\
\end{equation}
\begin{equation}\label{eq:eqB}
\int_{-\infty }^{\infty }f_{IX,IY}(x,y)\operatorname{d}\!x=\frac{f_{IY}(y)}{2}+\frac{f_{Y}(y)}{2}.\\
\end{equation}
\end{subequations}
\end{lemma}
\begin{proof}
Let
\begin{equation}
I_1 \defeq \int_{-\infty }^{\infty }f_{X,Y}(x,y)\cdot \log F_{X,Y}(x,y)\operatorname{d}\!x,
\end{equation}
which can be rewritten as
\begin{equation}
I_1 = \int_{-\infty }^{\infty }\log  F_{X,Y}(x,y). \frac{\partial}{\partial x} \frac {\partial F_{X,Y}(x,y)}{\partial y}~\operatorname{d}\!x,
\end{equation}
Using pars integral in the definition, it follows that
\begin{equation}
I_1=\frac{\partial }{\partial y} F_{X,Y}(x,y).\left [ \log ~F_{X,Y}(x,y) \right ] \biggr |_{-\infty }^{\infty }-\frac{\partial }{\partial y}\int_{-\infty }^{+\infty }dF_{X,Y}(x,.)
\end{equation}
that is,
\begin{equation}
I_1=\frac{\partial }{\partial y} F_{X,Y}(+\infty,y).\left [ \log ~F_{X,Y}(+\infty,y) \right ] -\frac{\partial }{\partial y}F_{X,Y}(+\infty,y)
\end{equation}
or finally
\begin{equation}
I_1=f_{Y}(y).\left [ \log ~F_{Y}(y) \right ] -f_{Y}(y),
\end{equation}
and Eqn~\ref{eq:eqB} follows. The proof of Eqn~\ref{eq:eqA} is similar. 
\end{proof}

\begin{corollary}
\textit{The marginal probability densities shown in Eqn~\ref{eq:eqA}, Eqn~\ref{eq:eqB} comply:
\begin{subequations}
\begin{equation}
\frac{1}{2}f_{IA}(\alpha)+\frac{1}{2}f_{A}(\alpha)\geq 0,
\end{equation}
\begin{equation}
\int_{-\infty }^{+\infty }\frac{1}{2}f_{IA}(\alpha)+\frac{1}{2}f_{A}(\alpha)d\alpha=1,
\end{equation}
\end{subequations}
where either $\{A=X~and~\alpha=x\}$ or $\{A=Y~and~\alpha=y\}. ~~~\Box$} 
\end{corollary}

\begin{corollary}
\textit{The Eqn~\ref{eq:joint_pdf} defines a joint density, i.e. it is non-negative and normalized}
\begin{equation}\label{eq:eqC}
f_{IX,IY}(x,y)>0 ~and~\int_{-\infty }^{\infty }\int_{-\infty }^{\infty }f_{IX,IY}(x,y)\operatorname{d}\!xdy=1.
~~~\Box
\end{equation}
\end{corollary}
Addressing joint-variables calls for the analysis of statistical independency~\cite{Walpole}. Now we have the fundamental additivity property of information densities:
\begin{proposition}
\textit{If $X$ and $Y$ are independent random variables with joint pdf $f_{X,Y}(x,y)$, then
their joint information-weighted pdf $f_{IX,IY}(x,y)$ is}:
\begin{equation}
f_{IX,IY}(x,y)=f_Y(y)f_{IX}(x)+f_X(x).f_{IY}(y).
\end{equation}
\end{proposition} 
\begin{proof}
Given that $X$ and $Y$ are independent, both $f_{X,Y}$ and $F_{X,Y}$  are separable, and
\begin{equation}
f_{IX,IY}(x,y)=-f_X(x)f_Y(y).\left [ \log F_X(x)+\log F_Y(y) \right ],
\end{equation}
concluding the proof.
\end{proof}
The integration of the joint density in this case (assuming independence between the variable $X$ and $Y$) results in marginal densities exactly with the same expressions as in Lemma ~\ref{marginal}, 
$\frac{1}{2}f_{IA}(\alpha)+\frac{1}{2}f_{A}(\alpha)$, as expected.

\section{Symmetrization of the Balancing of Tails in the ``Weighted by Information'' Distribution}

Seeking an interpretation of this ``shaping''~in the original distribution in Def. \ref{def:left}, we see that it is weighted by the amount of information $-\log F_X (x)$, which is rather large in the tail to the left, while the weighting factor decreases as $x$ grows from minus infinity to infinity (a bias to the left). The easiest way to perceive this effect is by noting that uniform distribution is weighted on the left part of the distribution (see Fig. ~\ref{fig:fig1}). Thus, the exchanging $f_X(\cdot)$ emphasizes the left tail. A similar procedure could be used to ``enhance''~ the right tail. Here, let us denote by $\bar{F}_X(x) \defeq P\left ( X> x \right )$ and $F_X(x) \defeq P\left ( X\leq x \right )$. 
(N.B. through this paper, $\bar{F}$ denotes the complementary distribution or survivor function, $\bar{F}_X(x)=1-F_X(x)$, or the ``tail function'' $\bar{F}(x) \defeq F(x,\infty)~\forall x$ as in~\cite{foss}. Compare this with the Def. \ref{def:left}, which can now be interpreted as a left cumulative information distribution.
\begin{definition}\label{def:rigth}
\textit{(right tail information-weighted pdf) The right cumulative information distribution $f_{X_{right}}(x)$ is defined by}:
\begin{equation} \label{eq:righttail}
f_{X_{rigth}}(x) \defeq -f_X(x)\cdot \log (\bar{F}_X(x)). ~~~\Box
\end{equation}
\end{definition}	
This is the probability density weighted by the cumulative hazard $H(x) \defeq -\log  (1-F_X(x))$ of survival analysis~\cite{lee_wang}. It turns out that this definition also builds a probability density. There are also similar (unshown) results to those of Section 2, but this turn with a deviation to the right. 
It seems yet natural to introduce a symmetrization procedure so as to make corrections in both tails of the distribution. 
\begin{definition}\label{def:bilateral_tail}
\textit{(bilateral information) The two-sided information-weighted distribution $f_{X_{2tail}}(x)$ is defined by}:
\begin{equation} \label{eq:twotail}
f_{X_{2tail}}(x) \defeq -\frac{1}{2}f_X(x)\cdot \log \left ( F_X(x).\bar{F}_X(x) \right ) . ~~~\Box
\end{equation}
\end{definition}	
It is now easier to write $f_{X_{2tail}}(x)$ as a linear combination (symmetric) of information-weighted  density at the left and the right, i.e.
\begin{equation}\label{eq:2tail}
f_{X_{2tail}}(x)=\frac{1}{2} \left ( f_{X_{left}}(x)+f_{X_{rigth}}(x) \right ).
\end{equation}
It is straightforward demonstrate that this is also a undeniable probability density. Again, all integrals have been checked with the aid of Wolfram site~\cite{Wolfram}. 
In the case of continuous distributions with bilateral weighting, there are just two intersections between the distributions ($f_X$ and $f_{X2tail}$), namely:
\begin{equation}
x^*= F_X^{-1}\left [  \frac {1}{2} \left ( 1 \pm \frac {\sqrt{e^2-4}}{e} \right ) \right ].
\end{equation}
For instance, intersection of $U$ and $U_{2tail}$ densities are $x^* \approx 0.161$ and $x^* \approx 0.839$; for the $IExponential$ distribution, the two crossing points occur at $x^* \approx 0.176$ and $x^* \approx 1.824$ and so on.\\
In the case of the information-weighted uniform distribution, $IU(0,1)$, the (2-tail) probability distribution function CDF can be easily derived:
\begin{equation}\label{eq:CDF_U2tail}
F_{IU~2tail}(x)=\frac{1}{2}\left ((1-x)\cdot \log(1-x)-x\cdot \log(x) \right )+x,
\end{equation}
which is obtained by integration, and is plotted in Fig. ~\ref{fig:fig3}A. Here we find an unexpected connection with the wavelet theory: 
a CDF weighted uniform distribution information is identical to that of uniform distribution,  overlaid with a unicyclic wavelet (see~\cite{deO_Araujo}). Formally,
\begin{equation}
F_{IU~2tail}(x)=F_U(x)+\psi_U(x)=x+\psi_U(x),
\end{equation}
where $\psi_U(x)$ is the compactly supported wavelet defined within $[0,1]$ by:
\begin{equation} \label{eq:psiU}
\psi_U(x) \defeq -\frac{1}{2}x\cdot \log(x)+\frac{1}{2}(1-x)\cdot \log(1-x),
\end{equation}
whose sketch is shown in Fig. ~\ref{fig:fig_wavelet}A. Its maximum (minumum) value is given by 
$\frac{1}{2}cotgh^{-1}\frac{e}{\sqrt{e^2-4}}-\frac{\sqrt{e^2-4}}{e}= 0.0733805...$ (-0.0733805...),
which occours at $\frac{1}{2}-\frac{\sqrt{e^2-4}}{2e}=0.161378...$ (1-0.161378...). The decomposition is illustrated in the Fig. ~\ref{fig:fig_wavelet}B. \\
Let us now examine one use of this decomposition. How to find the \textit{percent point function} (PPF or quantile function) associated with the CDF $F_{IU~2tail}$? 
Calculate the inverse of the function described in Eqn~\ref{eq:CDF_U2tail}, at first glance, does not seem trivial. But it is clear that reversing sign in half-cycles of the wavelet provides the inverse function sought.
Thus, the PPF $G(\alpha)$ of the distribution $U_{2tail}$ is:
\begin{equation}
x=G(\alpha) = F_U(\alpha)-\psi_U(\alpha)=\alpha+\frac{1}{2}\alpha\cdot \log(\alpha)-\frac{1}{2}(1-\alpha)\cdot \log(1-\alpha).
\end{equation}
Clearly, $G \circ F_{IU~2tail}$ is the identity, where the symbol $\circ$ denotes ``composition of functions''. 
\begin{figure}
\centering
       \begin{subfigure}{0.4\textwidth}
       \includegraphics[height=4cm]{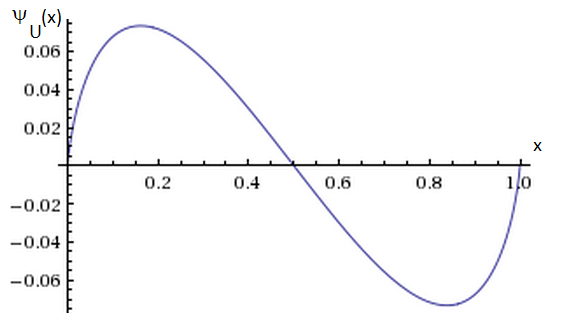}
       \captionof{figure}{Compactly supported wavelet $\psi_U(x)$ associated with the info-weighted uniform distribution. }
       \end{subfigure}%
       \hfill
       \begin{subfigure}{0.4\textwidth}
       \includegraphics[height=4cm]{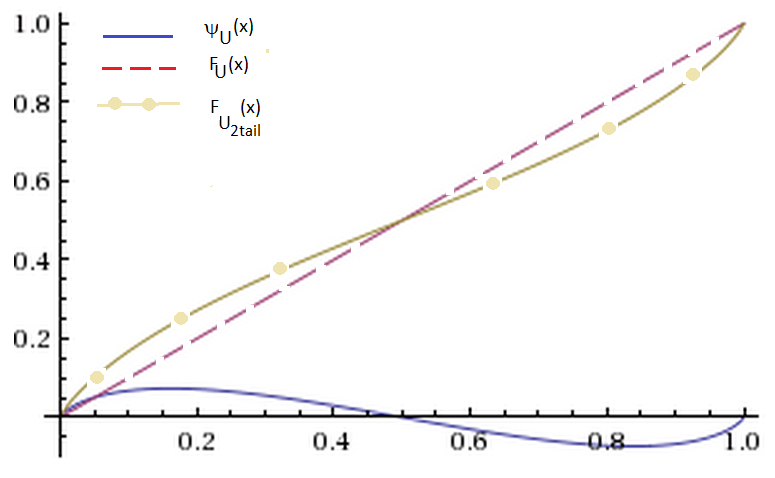}
       \captionof{figure}{Synthesis of the CDF of a variable $IX \sim U{2tail}(0,1)$ engendred by the uniform distribution (dashed line).}
       \end{subfigure}
\caption{Decomposition of the Bilateral-weighted Uniform Probability Distribution Function.}
\label{fig:fig_wavelet}
\end{figure}
Another case where the two-sided weigthed CDF has a simple close expression is the logistic, $ILogistic$ (also shown in Fig. ~\ref{fig:fig3}B):
\begin{equation}
F_{Ilogistic~2tail}(x)=\frac{1}{2}-\frac{1}{2}\frac{xe^x-\left ( e^x-1 \right )\log \left ( e^x+1 \right )+2}{e^x+1},
\end{equation}
\\
\begin{figure}
\centering
       \begin{subfigure}{0.4\textwidth}
       \includegraphics[height=4cm]{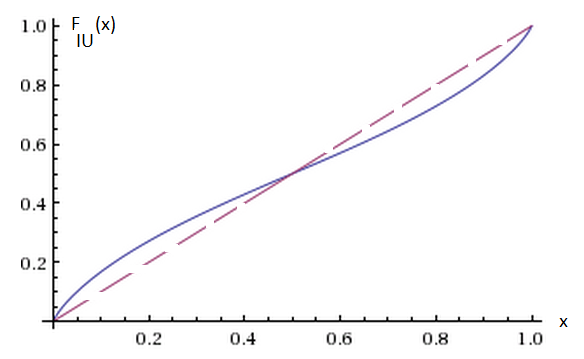}
       \captionof{figure}{Bilateral-weighted Probability Distribution Function CDF of a variable $IX \sim IU_{2tail}(0,1)$ engendred by the uniform distribution $U(0,1)$ (dashed line). }
       \end{subfigure}%
       \hfill
       \begin{subfigure}{0.4\textwidth}
       \includegraphics[height=4cm]{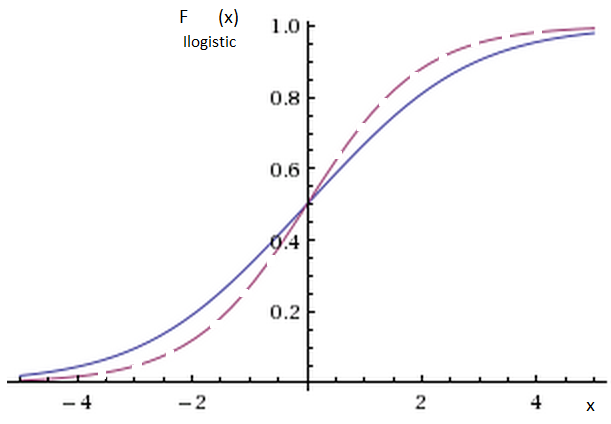}
       \captionof{figure}{Bilateral-weighted Probability Distribution Function CDF of a variable $IX \sim Ilogistic_{2tail}(0,1)$ engendred by the logistic distribution (dashed line).}
       \end{subfigure}
\caption{Bilateral heavy-tail cumulative probability function associated with uniform and logistic distributions. Dashed lines corresponds to original distributions.}
\label{fig:fig3}
\end{figure}
The corresponding two-sided weighted probability distribution are compiled in Table \ref{table:tab5}. Plots of the probability density of the new two-sided weighted distributions are shown in Fig. ~\ref{fig:fig4}. 
For unimodal distributions with asymmetry, the generation of distribution with weighting-information seems to give rise to a \textit{sag} around the median, with a different behavior for each side, as in Figure ~\ref{fig:fig4} (E, H, I and K). It is also worth to show that the mean of the two-tail weighted distribution remains unchanged in cases where the original distribution is symmetrical (see for example Figure ~\ref{fig:fig4} A, B, D, G). In these cases, 
the left and right-shifts cancels out each other, ($EX_{2tail}=E(X)-\Delta \mu +\Delta \mu=E(X)$ instead of Eqn~\ref{eq:mean_shift}).\\
Another interesting study concerns the evaluation of the degree of kurtosis associated with the distributions here introduced. Instead of using the \textit{excess kurtosis}, we chose to evaluate the percentile coefficient of kurtosis, $\kappa$ (\cite{abramowitz}), which is compared to the value 0.263 (normal curve):
\begin{equation}
\kappa \defeq \frac{1}{2} \frac {Q_3-Q_1}{P_{90}-P_{10}}.
\end{equation}
When evaluating the degree of kurtosis for distributions generated from the normal distribution, specifically, $IN (0,1)$ and $IN_{2tail}(0,1)$, we have:\\
\\
\centerline{$left$: $Q_1 \approx-1.493$, $Q_3\approx-0.299$, $P_{10}\approx-2.044$ and $P_{90}\approx0.221$.}\\
\centerline{$2tail$: $Q_1 \approx-0.927$, $Q_3\approx0.927$, $P_{10}\approx-1.649$ and $P_{90}\approx 1.649$.}\\
\\
Therefore, we obtain $\kappa\approx 0.2635$ and $\kappa \approx 0.2811$, respectively.
Despite having a left-weighted tail (and a slight asymmetry), it is seen in Figure ~\ref{fig:fig1}B that the distribution behavior remains practically mesokurtic 
(the fact is that extending on the left side is somewhat offset by tapering on the right side. This is best visualized by visual inspection of the ICauchy distribution.) 
However, observing Figure ~\ref{fig:fig4}B, it can be seen that the new bilateral distribution is platykurtic, with long tails.\\
\\
Aiming to determine closed expressions for the heavy-tailed probability distributions associated with a given CDF $F_X (x)$, 
the expressions defined in Eqn~\ref{eq:lefttail}, Eqn~\ref{eq:righttail} and Eqn~\ref{eq:twotail} were integrated, resulting in:
\begin{subequations} \label{eq:INFO_CDF}
\begin{equation}
F_{left}(x)=F_X(x)-F_X(x)\cdot \log F_X(x),
\end{equation}
\begin{equation}
F_{right}(x)=F_X(x)+\bar{F}_X(x)\cdot \log\bar{F}_X(x),
\end{equation}
\begin{equation}
F_{2tail}(x)=F_X(x)-\frac{1}{2}F_X(x)\cdot \log F_X(x)+\frac{1}{2}\bar{F}_X(x)\cdot \log\bar{F}_X(x).
\end{equation}
\end{subequations}
Here is a beautiful interpretation in the case of bilateral weighted distributions: 
Equation~\ref{eq:INFO_CDF}c which calculates the CDF of $X_{2tail}$ can be rewritten as:
\begin{equation}
F_{2tail}(x)=F_{IU~2tail}(F_X(x))=F_{IU~2tail} \circ F_{X}(x).
\end{equation}
Figure ~\ref{fig:INFO_CDF} was sketched so as to check the asymptotic behavior for these probability distribution (Eqn~\ref{eq:INFO_CDF}), considering the normal distribution and their respective info-weighted distributions.

\begin{table}[!h]
\centering
\caption{Bilateral information-weighted probability densities: 2-heavy-tail distributions derived from selected known probability distributions.}
\label{table:tab5}
\begin{tabular}{ c c } 
\hline 
distribution & $f_{X_{2tail}}(x)$ \\ \hline
\\
$U_{2tail}(0,1)$ & $\frac{-1}{2}\log \left (x.(1-x)\right )$\\
$N_{2tail}(0,1)$ & $-\frac{1}{2 \sqrt{2\pi }}e^{-x^2/2}\log \left ( \frac{1-erf^2(\frac{x}{\sqrt{2}})}{4}  \right )$\\
$E_{2tail}(1)$ & $-\frac{1}{2}e^{-x} \log \left ( e^{-x} - e^{-2x} \right )u(x)$ \\
$logistic_{2tail}$ & $-\frac{e^{-x} \left ( x-2\log (1+e^x) \right ) }{2 \left (1 -e^{-x} \right )^2}$\\
$Ray_{2tail}(1)$ & $\frac{-x}{2}e^{-x^2/2}\cdot \log \left ( e^{-x^2/2}-e^{-x^2} \right )u(x)$ \\  
$Par_{2tail}(\alpha,1)$ &$-\frac{\alpha \log  \left [ \left ( 1-x^{-\alpha} \right ). x^{-\alpha} \right ] }{2x^{\alpha+1}}$\\
$Cau_{2tail}(1)$ & $-\frac{1}{2 \pi(1+x^2)}\log  \left ( \left ( \frac{1}{2}+\frac{tg^{-1}(x)}{\pi} \right ).\left ( \frac{1}{2}-\frac{tg^{-1}(x)}{\pi} \right ) \right )$ \\
$Wei_{2tail}(1,k)$ & $\frac{-k}{2}x^{k-1}e^{-x^{k}}\log  \left ( \left ( 1-e^{-x^{k}} \right ).\left ( e^{-x^{k}} \right )\right )u(x)$\\
$Maxw_{2tail}(1)$& ${\frac{-x^2}{\sqrt{2\pi}}}e^{-x^{2}/2}\log  \left [ \left ( erf\left ( \frac{x}{\sqrt{2}}\right )-\sqrt{\frac{2}{\pi}} xe^{-x^{2}/2} \right ) . \left ( 1-erf\left ( \frac{x}{\sqrt{2}}\right )+\sqrt{\frac{2}{\pi}} xe^{-x^{2}/2} \right ) \right ]u(x)$ \\
$Kum_{2tail}(a,b)$ & $-\frac{ab}{2}.x^{a-1} \left ( 1-x^{a} \right )^{b-1}\cdot \log \left [ \left ( 1-\left ( 1-x^{a} \right )^{b} \right ).\left ( \left ( 1-x^{a} \right )^{b} \right ) \right ]$ \\
\\ \hline
\end{tabular}
\end{table}

\begin{figure}
       \centering
       \includegraphics[height=5.0cm]{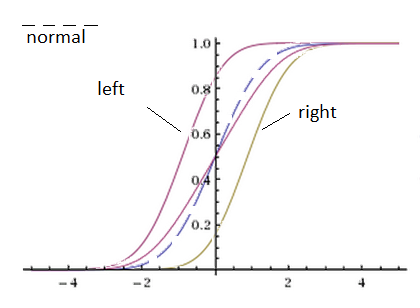}
       \captionof{figure}{CDF of information-weighted distributions conjugated with the normal $N(0,1)$ (dashed): \textit{left, right, and two-tail}. }.
\label{fig:INFO_CDF}
\end{figure}

By completeness, Table \ref{table:EV} shows information-weighted distributions conjugated with extreme value distributions~\cite{Werner}.
\begin{table}[!h]
\centering
\caption{Information-weighted Extreme-Value (EV) distributions: type I, II and III EV distributions.}
\label{table:EV}
\begin{tabular}{ c c c } 
\hline 
type & EV & $IEV_{2-tail}$\\ \hline
Gumbel  & $e^{-x-e^{-x}}$ & $-\frac{1}{2}e^{-x-e^{-x}}\cdot \log \left( e^{-e^{-x}}.\left[ 1-e^{-e^{-x}} \right] \right)$\\
Fr\'echet, $\alpha>0$ & $\alpha.x^{(\alpha+1)}.e^{-x^{-\alpha}}u(x)$ & $-\frac{\alpha}{2}.x^{-(\alpha+1)}.e^{-x^{-\alpha}}\cdot \log \left( e^{-x^{-\alpha}}.\left[ 1-e^{-x^{-\alpha}} \right] \right)u(x)$\\
Weibull, $k>0$ & $k.x^{k-1}.e^{-x^k}u(x)$ & $-\frac{1}{2}.k.x^{k-1}.e^{-x^k}\cdot \log \left( e^{-x^{-k}}.\left[ 1-e^{-x^{-k}} \right] \right)u(x)$
\\ \hline
\end{tabular}
\end{table}

One of the lines that deserves further attention is the application of these new distributions in Importance Sampling~\cite{Smith}. The proposal is to evaluate $E\left ( h(x) \right )$ by calculating the expected value ($E_I(\cdot)$) with respect to the tail-heavy associated distribution, i.e. via $E_{I} \left ( -h(x)/\log F_X(x)) \right )$. This should be still the subject of future research.

%plot 
\begin{figure}
       \centering
       \begin{subfigure}{0.3\textwidth}
       \includegraphics[height=3.3cm]{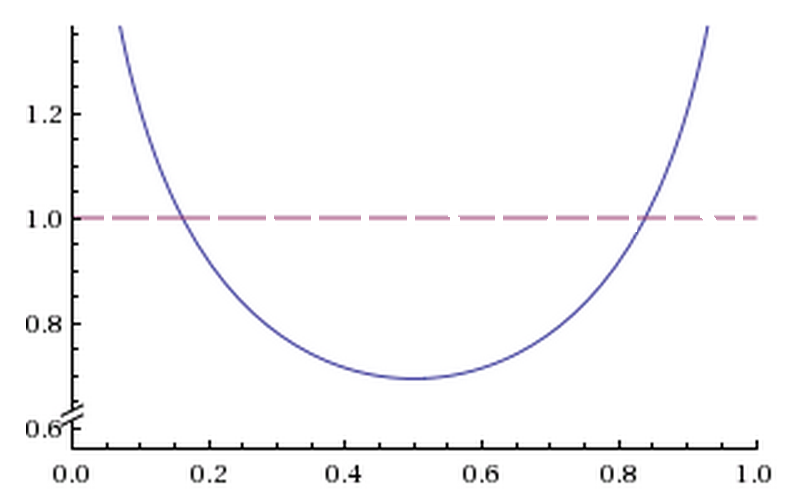}
       \captionof{figure}{$Uniform_{2tail}$}
       \end{subfigure}%
       \hfill
       \begin{subfigure}{0.25\textwidth}
       \includegraphics[height=3.3cm]{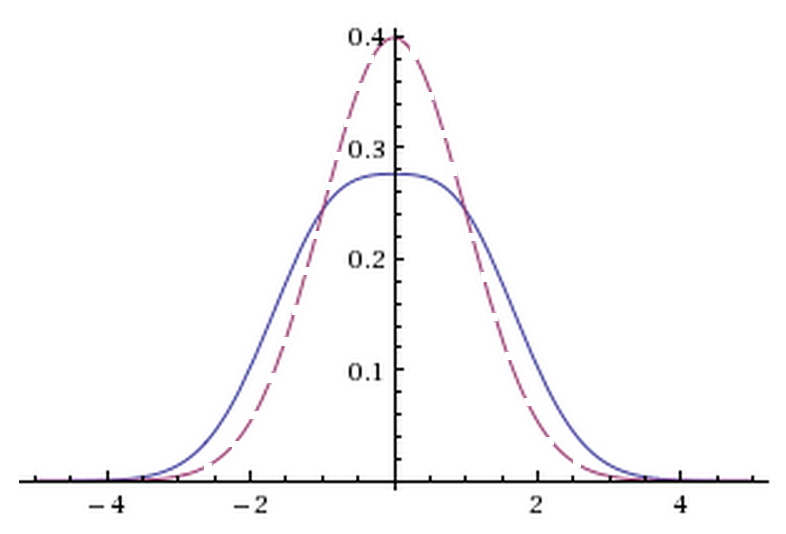}
       \captionof{figure}{$Normal_{2tail}$}
       \end{subfigure}

       \begin{subfigure}{0.20\textwidth}
       \includegraphics[height=3.3cm]{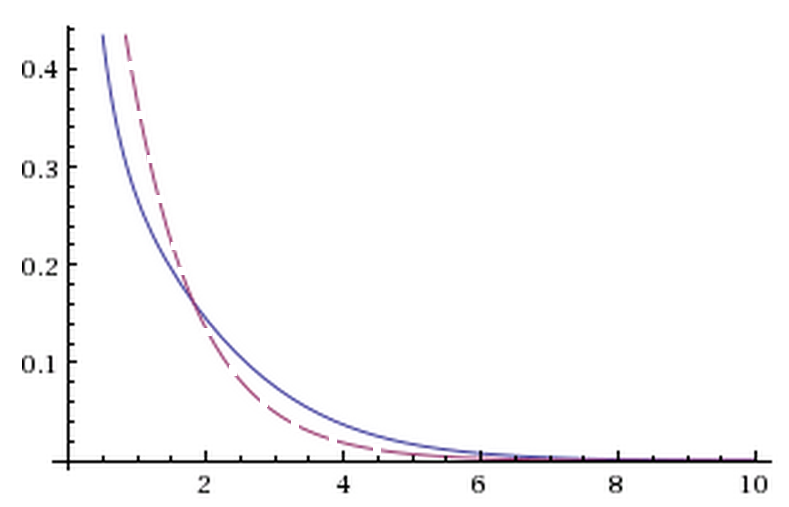}
       \captionof{figure}{$Exponential_{2tail}$}
       \end{subfigure}%
        \hfill
       \begin{subfigure}{0.2\textwidth}
       \includegraphics[height=3.3cm]{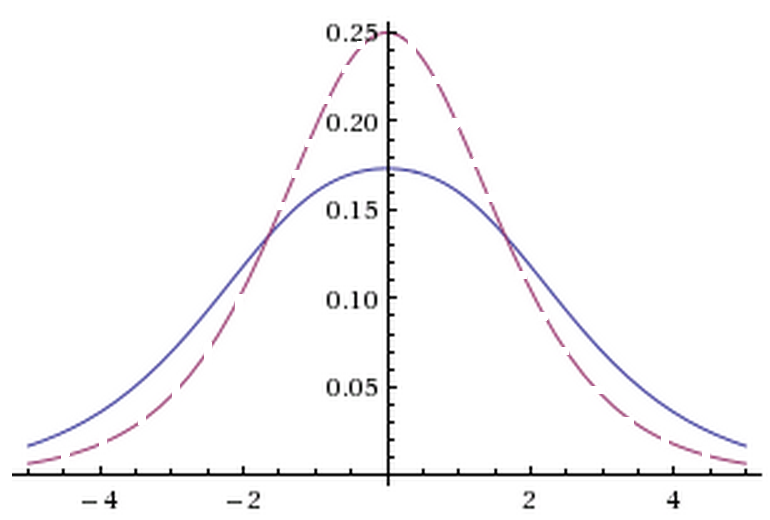}
       \captionof{figure}{$logistic_{2tail}$}
       \end{subfigure}

       \begin{subfigure}{0.2\textwidth}
       \includegraphics[height=3.3cm]{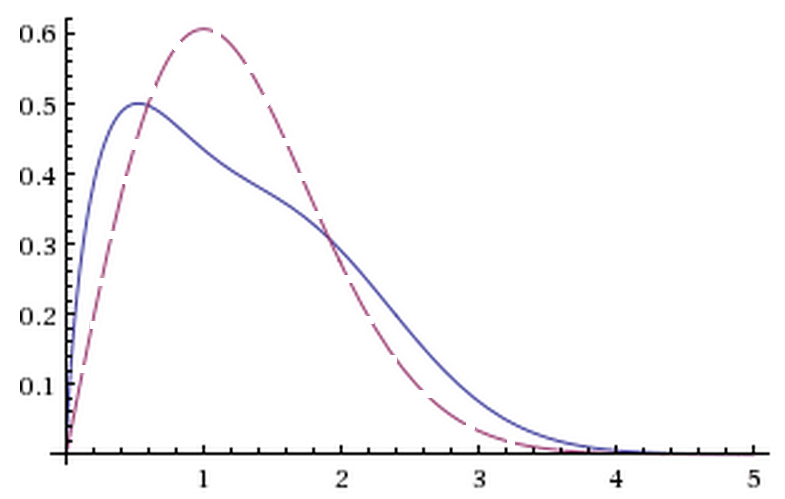}
       \captionof{figure}{$Rayleigh_{2tail}$}
       \end{subfigure}%
        \hfill
       \begin{subfigure}{0.20\textwidth}
       \includegraphics[height=3.3cm]{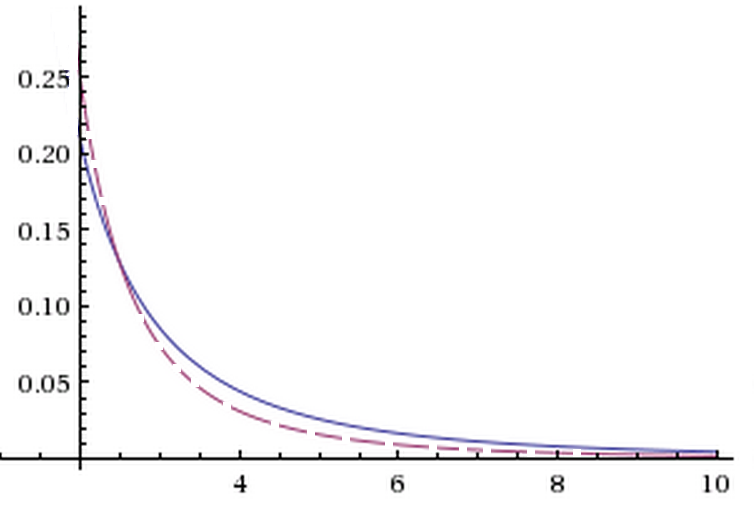}
       \captionof{figure}{$Pareto_{2tail}$}
       \end{subfigure}

       \begin{subfigure}{0.20\textwidth}
       \includegraphics[height=3.3cm]{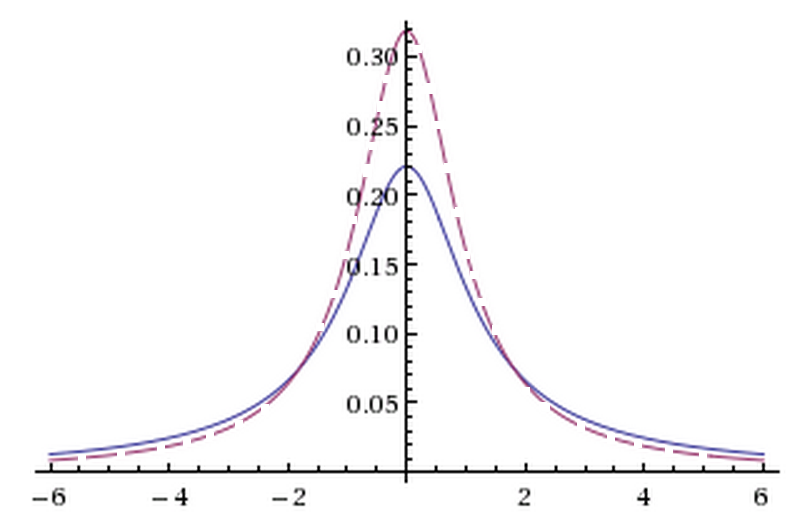}
       \captionof{figure}{$Cauchy_{2tail}$}
       \end{subfigure}%
        \hfill
       \begin{subfigure}{0.20\textwidth}
       \includegraphics[height=3.3cm]{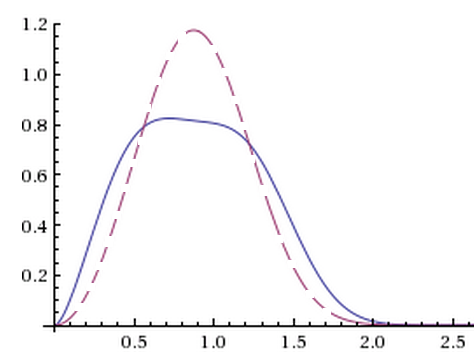}
       \captionof{figure}{$Weibull_{2tail}$}
       \end{subfigure}

       \begin{subfigure}{0.2\textwidth}
       \includegraphics[height=3.3cm]{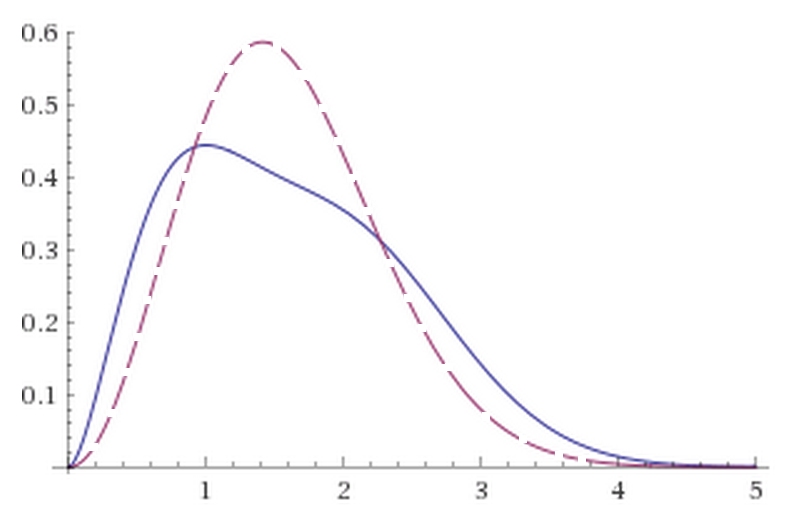}
       \captionof{figure}{$Maxwell_{2tail}$}
       \end{subfigure}%
        \hfill
       \begin{subfigure}{0.20\textwidth}
       \includegraphics[height=3.3cm]{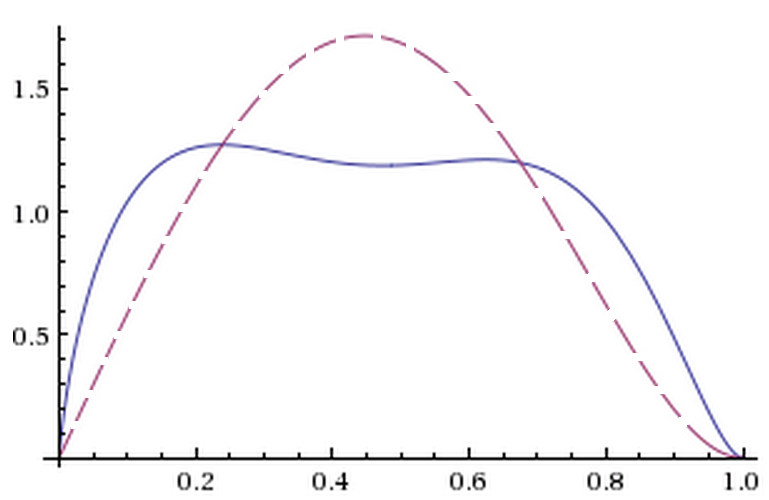}
       \captionof{figure}{$Kumaraswamy_{2tail}$}
       \end{subfigure}
\caption{Bilateral heavy-tail probability density function associated with standard distributions. Parameters: Uniform (0,1), Normal (0,1), Exponential (1), Rayleigh (1), Pareto (2,1),  Weibull (1,3), Maxwell-Boltzmann (1), Kumaraswamy (2,3). Dashed lines corresponds to original distributions.}
\label{fig:fig4}
\end{figure}

\section{Study on the Behavior of Distribution Tails}

This short section addresses the asymptotic behavior of the (unbounded supported) probability distributions here introduced. 
We have adopted the following definitions to investigate the distributions tail behavior (subexponential distributions):

\begin{definition}\label{def:heavy-tailed}
(\textit{heavy-tailed distribution}~\cite{Brown-Tukey},~\cite{foss})\\
i) A probability distribution is said to be right heavy-tailed if and only if
\begin{equation}
\lim_{x \to +\infty } e^{\lambda x}P\left ( X> x \right )=e^{\lambda x}\bar{F}_X(x)=+\infty ~~~(\forall \lambda >0).
\end{equation}
ii) A probability distribution is said to be left heavy-tailed if and only if
\begin{equation}
\lim_{x \to -\infty } e^{-\lambda x}P\left ( X\leq x \right )=e^{-\lambda x}F_X(x)=+\infty ~~~(\forall \lambda >0). ~~~\Box
\end{equation}
\end{definition}	
(indeed, the formal relationship could be $\int_{-\infty}^{\infty} e^{\lambda x}dF_X=+\infty.$)
The application of L'H\^{o}pital's rule~\cite{Protter} leads to a similar relationships with probability density, i.e.
\begin{equation}
\begin{split}
&\lim_{x \to +\infty } e^{\lambda x}.f_X(x)=+\infty ~~~(\forall \lambda >0)~and\\
&\lim_{x \to -\infty } e^{-\lambda x}.f_X(x)=+\infty ~~~(\forall \lambda >0).\\
\end{split}
\end{equation}
This is equivalent to the Lemma 2.7, Page 9, in~\cite{foss}, but following the \textit{lex parsimoni\ae} of Occam's Razor~\cite{forster}.
\begin{proposition}
The probability density functions $f_{X_{rigth}}(x),~ f_{X_{left}}(x),~and~f_{X_{2tail}}(x)$ describe heavy-tailed distributions.
\end{proposition}
\begin{proof}
Let us start examining $e^{\lambda x}.P(X_{rigth}>x)$ for each arbitrary $\lambda >0$. Thus,
\begin{equation}
e^{\lambda x} \int_{x}^{+\infty } f_{X_{right}}(\zeta)d\zeta= -e^{\lambda x} \int_{x}^{+\infty }f_X(\zeta)\log \left (1-F_X(\zeta)  \right )d\zeta.
\end{equation}
The following upper bound is promptly derived:
\begin{equation}
e^{\lambda x}.P(X_{rigth}>x) \geqslant -e^{\lambda x} \log \left (1-F_X(x)  \right ) \int_{x}^{+\infty} f_X(\zeta) d\zeta,
\end{equation}
that is
\begin{equation}
e^{\lambda x}.P(X_{rigth}>x) \geqslant e^{\lambda x}\bar{F}(x). \left [ -\log  \left (1-F_X(x) \right ) \right ]
\end{equation}
so that taking the limit $x \to +\infty$, the right-side of the previous equation diverges, even though $F$ is light-tailed.
Considering now $e^{-\lambda x}.P(X_{left}<x)$ for each arbitrary $\lambda >0$, we have
\begin{equation}
e^{-\lambda x} \int_{-\infty}^{x} f_{X_{left}}(\zeta)d\zeta= -e^{\lambda x} \int_{-\infty}^{x} f_X(\zeta)\log \left (F_X(\zeta)  \right )d\zeta
\end{equation}
and the upper bound:
\begin{equation}
e^{-\lambda x}.P(X_{left}<x) \geqslant -e^{-\lambda x}F_X(x) . \log \left [ F_X(x) \right ]
\end{equation}
holds. Therefore, assuming that $F_X(x)=\mathcal{O}(e^{-\kappa |x|})$ for some $\kappa>0$ and taking the limit $x \to -\infty$, it follows that:
\begin{equation}
\lim_{x \to -\infty } e^{-\lambda x}.P(X_{left}<x) =+\infty.
\end{equation}
The proof for $f_{X_{2tail}}(x)$ follows from applying Eqn~\ref{eq:2tail} and previous results derived for left and right tails.
\end{proof}

An interesting class of heavy-tailed distributions is that with regular variation in the tails (\cite{Goldie},~\cite{Werner}).
\begin{definition}\label{def:regularly}
\textit{A distribution with CDF is said to be regularly varying with index $\alpha>0$ iff}
\begin{subequations}
\begin{equation}
\lim_{x \to +\infty } \frac{\bar{F}_X(tx)}{\bar{F}_X(x)}=t^{-\alpha},~~~\forall t>0~(right).
\end{equation}
\begin{equation}
\lim_{x \to -\infty } \frac{F_X(tx)}{F_X(x)}=t^{-\alpha},~~~\forall t>0~(left). ~~~\Box
\end{equation}
\end{subequations}
\end{definition}
Indeed, L'H\^{o}pital's rule can be applied to show that
\begin{equation}
\lim_{x \to +\infty } \frac{f_X(tx)}{f_X(x)}=t^{-(\alpha+1)}.
\end{equation}
The parameter $\alpha$ is called tail index or extreme-value index. Let us now investigate the effect of the info-weighting on the tail index of a regular variation distribution. The following proposition can now be established:
\begin{proposition}
Both conjugated distributions, $X$ and $IX$, have the same tail index $\alpha$.
\end{proposition}
\begin{proof}
Suppose that $X$ has a regularly varying distribution with index $\alpha>0$, i.e. Def. \ref{def:regularly}a holds.
In order to evaluate whether or not $IX$ is also a regular variation distribution, we should compute:
\begin{equation}
\lim_{x \to +\infty } \frac{\bar{F}_{IX}(tx)}{\bar{F}_{IX}(x)}.
\end{equation}
Now, this is the same as evaluating
\begin{equation}
\lim_{x \to +\infty } \frac{t.f_{IX}(tx)}{f_{IX}(x)}.
\end{equation}
Applying Def. \ref{def:rigth} into the previous equation, we have
\begin{equation}
\lim_{x \to +\infty } \frac{\bar{F}_{IX}(tx)}{\bar{F}_{IX}(x)}=
t^{-\alpha}.\lim_{x \to +\infty } \frac{\log ~\bar{F}_{X}(tx)}{\log ~\bar{F}_{X}(x)}.
\end{equation}
The right-side limit can be evaluated using L'H\^{o}pital's rule again, giving
\begin{equation}
\lim_{x \to +\infty } \frac{\bar{F}_X(x)}{\bar{F}_X(tx)}. \frac{tf_X(tx)}{f_X(x)}=t^{\alpha}.\lim_{x \to +\infty }\frac{tf_X(tx)}{f_X(x)}=1.
\end{equation}
Finally,
\begin{equation}
\lim_{x \to +\infty } \frac{\bar{F}_{IX}(tx)}{\bar{F}_{IX}(x)}=t^{-\alpha}.
\end{equation}
\end{proof}

The concept of heavy-tail is primarily linked to the unbounded support distributions since it encompasses a $\lim_{x \to \pm \infty}$. 
But the information-weighting ``thickens'' the tail of the distributions, even those double bounded support. 
This can be understood by observing the behavior of $U$ and $IU$ distributions in Figure ~\ref{fig:fig4}A.\\
\\
The value of the survival function evaluated at the cross point between the two related distributions, $\bar{F}_X(x^{*})$, can be used to measure the the tail heaviness (Table \ref{table:arc_length}). 
Another possible way to assess its effects on the distribution edges is to examine the arc lenght of the distribution curve. For an arbitrary distribution $F_X$ with support $[a,b]$, the arc length is given by:
\begin{equation}
arc~lenght \left ( F_X \right ) = \int_{a}^{b} \sqrt{1+ f_X^2(x) } \operatorname{d}\!x.
\end{equation}
In order to investigate the tail, the arc length can be restricted to the region beyond a percentil (say 90\%), i.e.,
\begin{equation}
\int_{F_X^{-1}(0.9)}^{b} \sqrt{1+f_X^2(x) } \operatorname{d}\!x,~and/or~\int_{a}^{F_X^{-1}(0.1)} \sqrt{1+f_X^2(x) } \operatorname{d}\!x.
\end{equation}
The next table shows the calculated arc length to the compact support distributions studied here.

\begin{table}[!h]
\centering
\caption{Arc length for compact supported distributions (approximated values). Value of the 90\% percentile (PPF) are $x_U=0.90000~~x_{IU}=0.95035~~x_{Kum}=0.7321~~x_{IKum}=0.7953$. 
Area under the tail is $\bar{F}_{X} (x^{*})$, values of $x^*$ (0.83862 uniform, 0.67495 Kumaraswamy) are the crossing point between a distribution and its two-tail info-conjugated.} \label{table:arc_length}
\begin{tabular}{ c c c c} 
\hline \\
distribution & arc length & $>90\%$-tail arc length & $ \bar{F}_{X} (x^{*})$\\ \hline
$X\sim~U(0,1)$ & $1.41421$ & $0.141421$ & $0.161378$\\
$X\sim~U_{2tail}(0,1)$ & $1.43633$ & $0.11206$ & $0.234758$\\ \hdashline
$X\sim~Kum(2,3)$ & $1.48334$ & $0.295092$ & $0.161382$\\
$X\sim~Kum_{2tail}(2,3)$ & $1.44321$ & $0.235125$ & $0.234763$
\\ \hline
\end{tabular}
\end{table}
One sees that the information-weighted distributions have a heavier tail even for double-bounded distributions (as viewed in the second and third columns of the Table \ref{table:arc_length}). The increase in the area under the tail after the crossover of density curves was about 45\% for both uniform and Kumaraswamy distributions. The decrease in the arc length beyond 90\% percentile was circa 20\% in both cases.

\section{Concluding Remarks}

This paper presented new probability density functions associated with standard probability densities, which were called ``information-weighted density.'' ~The concept has also been extended to joint probability distributions and the consequences of the statistical independence between variables was investigated. A further inquiry should be carried out to compute moments that exist for variables with distributions $UI, \sim U_{2tail}, \sim IN, \sim N_{2tail}, ... \sim IKum, \sim Kum_{2tail}$ as well as other quantities such as their entropy (Shannon, R\'enyi etc.). This can be an innovative approach even to treating phenomena modeled by distributions in which the tails are not ``heavy'', but where they control the process performance (e.g. hypothesis tests, Monte Carlo simulation). The relationship between wavelets and CDFs derived from symmetric weighted distributions is another issue that deserves further attention. The new distributions of heavy tail with bilateral weighing are the focus of this work. To summarize, starting from an arbitrary probability distribution with CDF $F_X (x)$, a new distribution associated can be build, whose CDF is expressed by $F_X (x) + \psi_U(F_X (x))$ (see Eq~~ \ref{eq:INFO_CDF}c and Eq~~ \ref{eq:psiU}).\\
The designing of new densities \textit{bilaterally weighted by information} in both tails opens interesting perspectives for delving the examination of scenarios in which this approach can be used. Just to name a few: finance~\cite{kluppelberg}, insurance~\cite{mikosch}, computer systems~\cite{harchol}, SAR images~\cite{achim}, World Wide Web traffic~\cite{crovella}, geophysics~\cite{kohlbecker}, electricity prices~\cite{weron}. There seems to be a multitude of processes that fit into procedures of relevant tails~\cite{clauset}. Even though this article have little mathematical depth, a host of heavy-tailed (subexponential distributions) probability distributions can be generated from this approach, thereby providing more degrees of freedom for distribution choices for modeling a random phenomena of interest. Rather than a formal, elegant and rigorous presentation, the authors opted for an approach in Euler style, making it clear (no sweeping under the rug) the steps that led to the proposal.\\ 

\section*{Acknowledgements}
The first author thanks D.R. de Oliveira with whom he first shared an early version of this work.

\bibliographystyle{alea3}
\bibliography{bibi}

\end{document}